\documentclass{amsart}%
\usepackage[all]{xy}
\usepackage{amssymb}
\usepackage{amsfonts,amsthm,amsmath,amscd}
\usepackage{mathrsfs}

\makeatletter
\@namedef{subjclassname@2020}{\textup{2020} Mathematics Subject Classification}
\makeatother

\newcounter{zlist}

\newcounter{blist}

\newcounter{rlist}

\input xy
\xyoption{all}
\xyoption{2cell}
\UseAllTwocells

\swapnumbers
\newtheorem{theorem}{Theorem}[section]
\newtheorem{lemma}[theorem]{Lemma}
\newtheorem{thm}[theorem]{}
\newtheorem{proposition}[theorem]{Proposition}

\newtheorem{corollary}[theorem]{Corollary}
\newtheorem{remark}[theorem]{Remark}
\newtheorem{example}[theorem]{Example}
\numberwithin{equation}{section}

\newcommand{\ol}{\overline}

\newcommand{\N}{\mathbb{N}}
\newcommand{\E}{\mathbb{E}}

\newcommand{\D}{{\mathscr {D}}}

\newcommand{\xr}{\xrightarrow}
\newcommand{\ra}{\ensuremath{\xymatrix@1@C=16pt{\ar[r]&}}}
\newcommand{\dra}{\ensuremath{\xymatrix@1@C=25pt{\ar@{|->}[r]&}}}
\newcommand{\mra}{\ensuremath{\xymatrix@1@C=20pt{\ar[r] |<(0.4){\object@{|}}&}}}

\begin{document}

\title[Factorization of monoids]{SOME RESULTS ON FACTORIZATION OF MONOIDS}

\author[Zs. Balogh and T. Mesablishvili]{Zsolt Adam Balogh and Tamar Mesablishvili}

\address{Department of  Mathematical Sciences,	UAEU,  United Arab Emirates}
\email{baloghzsa@gmail.com}

\address{Department of Mathematics, I. Javakhishvili Tbilisi State University, Tbilisi, Georgia}
\email{tamar.mesablishvili392@ens.tsu.edu.ge}





\begin{abstract} Factorizations of monoids are studied. Two necessary and sufficient conditions in terms of so-called
descent 1-cocycles for a monoid to be factorized through two submonoids are found.  A full classification of those factorizations of a
monoid whose one factor is a subgroup of the monoid is obtained. The relationship between monoid factorizations and
non-abelian cohomology of monoids is analyzed. Some applications of semi-direct product of monoids are given.
\end{abstract}


\maketitle

\section {INTRODUCTION}
Factorizations of mathematical objects is an important topic in mathematics, whose basic underlying idea is to represent
a mathematical object as a product of two (usually simpler) sub-objects with minimal intersection.
The factorization problem can be divided into two parts. The first part is to find all the factorizations of a given mathematical object through its sub-objects. The second part is to study the properties of a mathematical object that has already been factorized with respect to its two smaller sub-objects.

In this paper we deal with the first part of the factorization problem for monoids. Our methodology is inspired by the paper \cite{BM}.
As a first step, we consider \emph{descent 1-cocycles} for monoids (\cite{Me}) and then give two necessary and sufficient
conditions in terms of descent 1-cocycles for a monoid to be factorized through two submonoids (see, Theorems \ref{fact.mon.}
and \ref{bicross}). Furthermore, we provide a full classification of those factorizations of a monoid whose one factor is a subgroup of the monoid (Theorem \ref{com.mon}). Next, Theorems \ref{bijection}, \ref{spl.com.} and \ref{T-compl.mon.} describe certain  relationships between monoid factorizations, non-abelian cohomology of monoids and semi-direct product
of monoids. Finally, we give examples for calculating how many ways a monoid can be factorized.

\section {MONOID FACTORIZATION AND DESENT COHOMOLOGY}\label{fdmon.}
Given a monoid $M$ with unit $\textsf{1}_M$ and elements $m_1,m_2 \in M$. We sometimes will write
 $m_1 \cdot m_2$ for $m_1m_2$ if it helps to avoid confusion.
To indicate that $A$ is a submonoid of the monoid $M$, we use the notation $A \leqslant M$.
In this case, we write $\imath_A$ for the canonical embedding $A \hookrightarrow M$.
If $X$ and $Y$ are subsets of $M$, then $XY$ is the set $\{xy : x\in X \,\,\text{and}\,\, y \in Y\}.$
Given another monoid $N$. Let $0_{M,N}$ denote the homomorphism $M \to N$ that sends each
element of $M$ to the unit element $\textsf{1}_N$ of $N$. This mapping is called the \emph{zero homomorphism} from
$M$ to $N$. Let us denote by $U(M)$ the set of invertible elements of $M$.
Throughout this paper $\mathbb{N}$ denotes the set of positive natural numbers $\{1,2, \ldots\}$ and let $\N_0=\{0\} \cup \N $.

\begin{thm}\em Let $A$ and $B$ be two submonoids of a monoid $M$. A \emph{monoid factorization} is a monoid $M$ together with its submonoids $A, B$ such that the multiplication map
\[[\imath_{A},\imath_{B}]: A \times B \to  M,\,\, (a,b) \longmapsto ab\]
is bijective, or equivalently, each element $m \in M$ can be written uniquely as a product $m=ab$ with $a\in A, \, b\in B$.
In this case one says that $M$ \emph{is factorized by its submonoids} $A$ and $B$, or that the couple $(A,B)$ is a
\emph{factorization} of $M$. $A$ and $B$ are respectively called the \emph{first} and \emph{second factors}
of the factorization. We shall sometimes use the term ``$(A,B \subseteq M)$ is a monoid factorization'' to mean that
the monoid $M$ is factorized by its submonoids $A$ and $B$. $(A,B)$ is a \emph{proper factorization of} $M$
if $A$, and then also $B$, is a proper submonoid of $M$. The monoid $M$ \emph{admits a factorization}, or is \emph{factorizable},
if $M$ is factorized by its proper submonoids. Otherwise it is called \emph{non-factorizable}.

\begin{example} \label{int.0} The set of positive natural numbers $\mathbb{N}$
under multiplication is a monoid. It may readily be verified that the set $\mathbb{O} = \{2k -1|
k \in \mathbb{N}\}$ of odd numbers is a submonoid of $\mathbb{N}$. Write $\mathcal{C}$ for the
cyclic submonoid \[\{2^i: i \,\, \text{is a non-negative integer}\}\subset \mathbb{N}.\]
Since every positive natural number can be uniquely written as a product of a non-negative integer
power of 2 and an odd integer, it follows  that the map
\[\mathbb{O}\times \mathcal{C} \to \mathbb{N}, \, (k,2^i)\mapsto 2^ik\] is bijective.
This shows that $(\mathbb{O},\mathcal{C})$ is a factorization of the monoid $\mathbb{N}$.
\end{example}

Let $M$ be a monoid. We write $\textsf{FAC}(M)$ for the set of all factorizations of $M$, i.e.
\[ \textsf{FAC}(M)=\{(A,B):A \leq M, B \leq M \,\,\text{and}\,\,  (A,B)\,\,\text{is a factorization of}\,\, M \}.\]
For a fixed submonoid $A$ of $M$, let $\textsf{FAC}(A/M)$ be denote the set of
those  submonoids $B \leq M$ for which the pair $(A,B)$ is a monoid factorization of $M$.
Note that when there already exists a monoid factorization of $M$ whose first factor
is $A$, then $\textsf{FAC}(A/M)$ may  be considered as a pointed set with the second factor
of the factorization as the base point. It is easy to see that
\[
\textsf{FAC}(M)=\bigsqcup_{A \subsetneqq M }\textsf{FAC}(A/M).
\]
\end{thm}

The following proposition gives a useful necessary condition for a submonoid of a given monoid to be the
first or second factor in a factorization of the monoid.

\begin{proposition}\label{first.factor} Let $(A,B)$ be a monoid factorization of $M$ and $a\in A$, $b\in B$ and $m\in M$. If
$am \in A $, then $m \in A$. Symmetrically, if $mb\in B$, then $m\in B$.
\end{proposition}
\begin{proof} According to the symmetry, it suffices to prove the result only for the first sentence. Suppose that $a \in A$ and  $m \in M $ are such that $am \in A $. Since $(A,B)$ is a monoid factorization
of $M$, there exists a unique pair $(a',b)\in A\times B$ such that $m=a'b$. Then
$am=aa'b$ and since $am, aa'\in A$, it follows from the uniqueness in the monoid factorization condition
that $b=\textsf{1}_M$. Thus $m=a'b=a' \in A$.
\end{proof}

The following example is an application of Proposition \ref{first.factor}.

\begin{example} \label{int.}Let $(\mathbb{Z},+,0)$ be the monoid of integers and $\N_0$
its submonoid of non-negative integers. Then $\emph{\textsf{FAC}}(\N_0/\mathbb{Z})=\emptyset$, since
$\N_0$ does not satisfy the condition of  Proposition \ref{first.factor}. Indeed, $10 \in
\N_0$ and $10+(-5)  \in \N_0$, but $-5 \notin \N_0$.
\end{example}

For a fixed factorization $(A,B)$ of a monoid $M$, we denote by $l^B_A$ and $ r^B_A$ (or, more simply, by $l$ and $r$
if no confusion can arise; we add superscripts and subscripts only if there is more than one monoid factorization in context)
the maps $M\to A$ and  $M\to B$ defined implicitly by
\[[\imath_A,\imath_B]^{-1}(m)=(l^B_A(m),r^B_A(m)), \qquad (m \in M).\]

\begin{proposition} \label{ker} Let $(A, B)$ be a factorization of a monoid $M$. Then
\begin{equation}\label{lr}
\emph{\textsf{Ker}}(l)=B
\quad \text{and}\quad
\emph{\textsf{Ker}}(r)=A.
\end{equation}
\end{proposition}
\begin{proof}We only prove the first equation, the second one can be proved similarly.
Let $(A, B)$ be a factorization of a monoid $M$. Then for any $b \in B$, $b=\textsf{1}_M b$ and we have
$l(b)=\textsf{1}_M $ for all $b \in B$. Thus, $B\subseteq \textsf{Ker}(l)$. Conversely, let $m \in M$ be such that
$l(m)=\textsf{1}_M $. Since $m=l(m)r(m)$, it follows that $m=r(m) \in B$. Hence $\textsf{Ker}(l)\subseteq B$ and
the proof is completed.
\end{proof}

\begin{proposition} \label{l} Let $(A, B)$ be a factorization of a monoid $M$. Then the map $$l=l^B_A: M \to A$$
satisfies the following conditions:
\begin{enumerate}
  \item [(L1)] $l(a)=a$ for all $a\in A$;
  \item [(L2)] $l(am)=al(m), \, $  for all $a\in A$ and $m \in M$;
  \item [(L3)] $l(m_1m_2)=l(m_1l(m_2))$, for all $m_1, m_2 \in M$.
\end{enumerate}
\end{proposition}
\begin{proof} Since $a=a\textsf{1}_M$ for all $a\in A$, the definition of $l$ shows that (L1) holds.
Next, since $am=al(m)r(m)$ and $al(m)\in A$, while $r(m) \in B$, by the definition of $l$ we have that
$l(am)=al(m)$, so (L2) holds. Finally, let $m_1,m_2$ be arbitrary elements of $M$. Clearly,
\begin{itemize}
  \item[-] $m_1=l(m_1)\cdot r((m_1)$;
  \item[-] $m_2=l(m_2)\cdot r(m_2)$;
  \item[-] $r(m_1)\cdot l(m_2)=ab$, where $a=l(r(m_1)\cdot l(m_2))$ and $b=r(r(m_1)\cdot l(m_2))$.
\end{itemize} As a consequence
\[m_1m_2=l(m_1)\cdot r(m_1)\cdot l(m_2)\cdot r(m_2)=l(m_1)\cdot a \cdot b \cdot r(m_2).\]
Since $l(m_1)\cdot a \in A$ and $b\cdot r(m_2) \in B$, it follows that
\[l(m_1m_2)=l(m_1)\cdot a.\]
On the other hand, we have:
\[m_1\cdot l(m_2)=l(m_1)\cdot r(m_1)\cdot l(m_2)=l(m_1)\cdot a \cdot b,\]
whence -- since $l(m_1)\cdot a \in A$ and $b\in B$ -- one concludes that
\[l(m_1l(m_2))=l(m_1)\cdot a.\] Consequently, $l(m_1m_2)=l(m_1l(m_2))$ and (L3) also holds.
\end{proof}

Symmetrically  we have the following.

\begin{proposition} \label{r} Let $(A, B)$ be a factorization of a monoid $M$. Then the map $$r=r^B_A: M \to B$$
satisfies the following conditions:
\begin{enumerate}
  \item [(R1)] $r(b)=b$ for all $b\in B$;
  \item [(R2)] $r(mb)= r(m)b, \, $  for all $m\in M$ and $b \in B$;
  \item [(R3)] $r(m_1m_2)=r(r(m_1) m_2)$, for all $m_1, m_2 \in M$.
\end{enumerate}
\end{proposition}

We now use Proposition \ref{l} to give an example showing that the condition in Proposition \ref{first.factor}
is not sufficient for a submonoid of a given monoid to be the first factor in a factorization of the monoid.

\begin{example} \label{int.1} The set of non-negative even integers $\E=2\N_0$
under addition is a monoid and hence a submonoid of $(\N_0,+,0)$. If $n$ is a non-negative
even integer and $m \in \N_0$ is such that $n+m$ is even, then clearly $m$ is also even. Hence
$\mathbb{E}$ satisfies the condition of  Proposition \ref{first.factor}. We claim that $\mathbb{E}$ cannot
appear as the first factor in a monoid factorization of $(\N_0,+,0)$; equivalently, $\emph{\textsf{FAC}}(\mathbb{E}/\N_0)=\emptyset$.
Indeed, if $\,\emph{\textsf{FAC}}(\mathbb{E}/\N_0)\neq\emptyset$, then by Proposition \ref{l}, there
exists a map $l:\N_0 \to \mathbb{E}$ satisfying Conditions (L1) - (L3). Since $\mathbb{E}$ is a commutative
monoid,  $l$ is a homomorphism of monoids (see, \cite[Proposition 3.9]{BM}). This implies in particular that
\[ 2 \stackrel{(L1)}=l(2)=l(1+1)=l(1)+l(1).\]
Keeping in mind that $l$ takes its values from $\mathbb{E}$ it follows that $l(1)=1$. Then  we have for all $n \in \N_0$:
\[l(2n+1)=l(2n)+l(1)\stackrel{(L1)}=2n+l(1)=2n+1,\]
which contradicts the assumption that $l$ maps $\N_0$ to $\mathbb{E}$. This shows that $\emph{\textsf{FAC}}
(\mathbb{E}/\mathbb{Z})=\emptyset$, as desired.
\end{example}

Let $X$ be a submonoid of a monoid $M$. We write $\mathcal{D}_l(M,X)$ (resp. $\mathcal{D}_r(M,X)$)
for the set of those maps $M \to X$ satisfying Conditions (L1) - (L3) (resp. (R1) - (R3)). The elements of
$\mathcal{D}_l(M,X)$ (resp. $\mathcal{D}_r(M,X)$) are called (\emph{left}) (resp. (\emph{right}) \emph{1-descent cocycles},
see \cite{Me} for more details. According to Propositions \ref{l} and \ref{r}, for any factorization $(A,B)$ of $M$, both sets
$\mathcal{D}_l(M,A)$ and $\mathcal{D}_r(M,B)$ are pointed with base points $l^B_A$ and $r^B_A$, respectively.

\begin{proposition}\label{left} Let $A$ be a submonoid of a monoid $M$. Then any $\,a_0 \in U(A)$ induces a bijection
\[a_0 \star -:\mathcal{D}_l(M,A) \to \mathcal{D}_l(M,A)\] defined by \[(a_0 \star q) (m)=q(ma_0)a^{-1}_0.\]
\end{proposition}
\begin{proof} Fix an element $a_0 \in U(A)$.  We first show that for an arbitrary element $q\in \mathcal{D}_l(M,A)$,
the map $a_0 \star q$ lies in $\mathcal{D}_l(M,A)$, that is, satisfies the conditions (L1) - (L3).
To begin with, observe that as it follows easily from the definition of $a_0 \star q$, it takes values in $A$. Next,
since
\[(a_0 \star q)(\textsf{1}_M)=q(\textsf{1}_Ma_0)a^{-1}_0=q(a_0)a^{-1}_0\stackrel{(L1)}=a_0a^{-1}_0=\textsf{1}_M,\]
it follows that the map $a \star q$ satisfies Condition (L1). For Condition (L2) we have
\[(a_0 \star q)(am)=q(ama_0)a^{-1}_0 \stackrel{(L2)}= aq(ma_0)a^{-1}_0=a((a_0 \star q)(m)).\]
It now remains to verify Condition (L3).
 \begin{align*}
(a_0 \star q)(m_1m_2) &=&\text{by definition of}\,\,\star \\
&=q(m_1m_2a_0)a^{-1}_0 &\text{by}\,\, (\textsc{L}3)\\
&=q(m_1q(m_2a_0))a^{-1}_0  \\
&=q(m_1q(m_2a_0)a^{-1}_0a_0)a^{-1}_0&\text{by definition of}\,\, \star \\
&= (a_0 \star q)(m_1q(m_2a_0)a^{-1}_0)&\text{by definition of}\,\,\star\\
&=(a_0 \star q)(m_1 (a_0 \star q)(m_2)).
\end{align*}

Finally, it is easy to see that the map $a_0 \star -$ is bijective with inverse
$a^{-1}_0 \star -$.
\end{proof}

Recall that a \emph{left action of a monoid $M$ on a set $X$} is a function
\[M \times X \to X,\,\,(m,x) \mapsto m \star x\] such that
\begin{itemize}
  \item [(A1)] $\textsf{1}\!_{M}\star x=x;$
  \item [(A2)] $(m_1 m_2)\star x=m_1 \star (m_2 \star x)$
\end{itemize} for all $m_1,m_2 \in M$ and all $x \in X.$

\begin{proposition}\label{left.m.} Let $A$ be a submonoid of a monoid $M$.
The association \[(a,q) \longmapsto a \star q \] is a left action of the group $U(A)$ on the set
$\mathcal{D}_l(M,A)$.
\end{proposition}
\begin{proof} By Proposition \ref{left}, it suffices to show that $(a,q) \longmapsto a \star q $
satisfies Conditions (A1) and (A2). There is no difficulty in verifying Condition (A1) and
the following calculation verifies Condition (A2):
\begin{align*}
(a_1 \star (a_2 \star q))(m) &=(a_2 \star q) (ma_1)\cdot a_1^{-1}\\
&=q(ma_1a_2)\cdot a_2^{-1}\cdot a_1^{-1}\\
&= q(ma_1a_2)\cdot(a_1a_2)^{-1}\\
&=((a_1a_2)\star q)(m).
\end{align*} Here the first, second and
fourth  equality follow from the definition of $\star$.
\end{proof}

The quotient of $\mathcal{D}_l(M,A)$ by the action
\[U(A)\times \mathcal{D}_l(M,A) \to \mathcal{D}_l(M,A),\,\, q \mapsto a \star q\]
is called the \emph{(left) first descent cohomology
set of $M$ with coefficients in $A$} (see \cite{BM} for more details) and is denoted by $\D_l(M,A)$. Thus,
two elements $q, q'\in \mathcal{D}_l(M,A)$ are \emph{equivalent} under the action of the group $U(A)$
if there exists an element $a \in U(A)$ such that $q(m)a=q'(ma)$ for all $m \in M$. If
there exist submonoids $A, B$ of $M$ such that the pair $(A,B)$ is a monoid factorization of $M$, then
$\D_l(M,A)$ (resp. $\D_r(M,B)$) is a pointed set whose point is the equivalence class of the map $l^B_A$
(resp. $r^B_A$).

We now return to general factorizations of monoids. Quite obviously, if $(A,B)$ is a factorization of a monoid $M$,
then $M=AB$ and $A\cap B= \{\textsf{1}_M\}$. While the converse holds if $M$ is a group and $A$ and $B$
are its subgroups (e.g., \cite{Dummit}), it is not true in general, as the following example shows.

\begin{example} \label{int.2} The set $-\N_0=\{0,-1,-2, \ldots\}$ of non-positive integers
under addition is a monoid and hence a submonoid of $(\mathbb{Z},+,0)$. It is easy to see that
$-\N_0\cap \N_0=\{0\}$ and that the map
\[-\N_0 \times \N_0\to \mathbb{Z},\,\, (m,n)\mapsto m+n\]
is surjective. But this map is not injective. Indeed, for $ -7,-3\in -\N_0$, and
$1,5 \in \N_0$, one has $-7+5=-3+1$.
\end{example}

The following result gives necessary and sufficient conditions for a pair of submonoids of
a given monoid to be a factorization of the monoid.

\begin{theorem}\label{fact.mon.} A monoid $M$ is factorized by its submonoids $A$ and $B$ if and only if
the following conditions hold:
\begin{itemize}
  \item [(i)] $AB=M$;
  \item [(ii)] there exists a map $l : M \to A$ satisfying Condition (L2) such that
$\emph{\textsf{Ker}}(l)=B$;
  \item [(iii)] there exists a map $r \in M \to B$ satisfying Condition (R2) such that
  $\emph{\textsf{Ker}}(r)=A$.
\end{itemize}
\end{theorem}
\begin{proof} If a monoid $M$ is factorized by its submonoids $A$ and $B$, then clearly
$AB=M$ and hence (i) holds. Next, $l=l^B_A:M\to A$ satisfying Condition (L2) by Proposition \ref{l}
and $\textsf{Ker}(l)=B$ by Proposition \ref{lr}. Thus, (ii) holds. Similarly,
$r=r^B_A:M \to B$ satisfying Condition (R2) by Propositions \ref{r} and $\textsf{Ker}(r)=A$
by Proposition \ref{lr}. Thus, (iii) also holds.

Conversely, suppose that Conditions (i), (ii) and (iii) hold. Condition  (i) guarantees that
the map $[\imath_A,\imath_B]: A \times B \to M$ is surjective. We claim that this
map is injective. Indeed, if $(a,b), (a',b')\in A \times B$ are such that $[\imath_A,\imath_B](a,b)=
[\imath_A,\imath_B] (a',b')$, then $ab=a'b'$. Since $\textsf{Ker}(l)=B$ and $\textsf{Ker}(r)=A$,
it follows that
$$l(b)=l(b')=\textsf{1}_M \,\,\, \text{and} \,\,\, r(a)=r(a')=\textsf{1}_M.$$ We then have
$$a=al(b)\stackrel{(L2)}=l(ab)=l(a'b')\stackrel{(L2)}=a'l(b')=a'$$
and
$$b=r(a)b\stackrel{(R2)}=r(ab)=r(a'b')\stackrel{(R2)}=r(a')b'=b'.$$ Thus, the map
$[\imath_A,\imath_B]$ is injective and so it is bijective.
Consequently, $M$ is factorized through $A$ and $B$.
\end{proof}

The next result gives another necessary and sufficient condition for a monoid to be
factorized by its submonoids. It should be compared with \cite[Theorem 2]{Mes}.
(Given a map $f:X \to Y$ of sets, we write $\textsf{K}[f]$ for the set
$\{(x_1,x_2)\in X \times X : f(x_1)=f(x_2)\}$.)

\begin{theorem}\label{bicross} A monoid $M$ is factorized by its submonoids $A$ and $B$
if and only if there are maps $A  \xleftarrow{l} M \xr{r} B$ of sets such that:
\begin{itemize}
  \item [(i)] $l$ satisfies Condition (L2);
  \item [(ii)] $r$ satisfies Condition (R2);
  \item [(iii)] $r \imath_A =0_{A,B}$ and  $l \imath_B =0_{B,A}$;
  \item [(iv)] $\emph{\textsf{K}}[l]\cap \emph{\textsf{K}}[r]=\emph{\textsf{K}}[\emph{\textsf{Id}}_M]$.
\end{itemize}
\end{theorem}
\begin{proof} "If": If $M$ is factorized through its submonoids $A$ and $B$, then the maps
$l=l^B_A$ and $r=r^B_A$ satisfy Conditions (i) and (ii) by Proposition \ref{fact.mon.}. Moreover, $\textsf{Ker}(l)=B$ and
$\textsf{Ker}(r)=A$ by Proposition \ref{ker}.  This implies that Condition (iii) is also satisfied.

For Condition (iv), suppose that $m, m' \in M$ are such that $(m,m')\in \textsf{K}[l]\cap \textsf{K}[r]$.
Then $l(m)=l(m')$ and $r(m)=r(m')$ and we have:
\[
m=l(m)r(m)=l(m')r(m')=m'.
\]
Therefore, $\textsf{K}[l]\cap \textsf{K}[r]=\textsf{K}[\textsf{Id}_M].$

"Only if": In the light of Proposition \ref{fact.mon.}, we have only to prove that
\begin{itemize}
  \item[-]  $\textsf{Ker}(l)=B$;
  \item[-]  $\textsf{Ker}(r)=A$;
  \item[-]  $AB=M$.
\end{itemize}
In order to prove these equalities, observe first that Condition (iii) guarantees that there
are inclusions  $B\subseteq \textsf{Ker}(l)$ and  $A\subseteq \textsf{Ker}(r)$. Next, since
$r(m)\in B\subseteq \textsf{Ker}(l)$ and $l(m)\in A\subseteq \textsf{Ker}(r)$ for
any $m \in M$, we have
\[ l(l(m)r (m)) \stackrel{(L2)}=l(m)l(r (m))=l(m)\textsf{1}_M= l(m) \]
and
\[r(l(m)r(m)) \stackrel{(R2)}=r(l(m))r(m)=\textsf{1}_Mr(m)= r(m).\]
It follows that the pair $(l(m)r(m), m)$ lies in  $\textsf{K}[l] \cap \textsf{K}[r]$. Then $l(m)r(m)=m$
by (iv). Since $m \in M$ was an arbitrary element, this implies that $AB=M$.

Finally, if $m \in M$ is such that $l(m)=\textsf{1}_M$ (resp. $r(m)=\textsf{1}_M$), then
$m=l(m)r(m)=r(m)$ and hence $m$ lies in $B$
(resp. $m=l(m)r(m)=l(m)\in A)$. Therefore, $ \textsf{Ker}(l) \subseteq B$ (resp. $ \textsf{Ker}(r) \subseteq A$).
Consequently, $\textsf{Ker}(l)=B$ and $\textsf{Ker}(r)=A$. This completes the proof of the theorem.
\end{proof}

In order to proceed, we need the following

\begin{proposition}\label{ker.mon} Let $A$ be a submonoid of a monoid $M$. Then for any
$l \in \mathcal{D}_l(M,A)$, the set $\emph{\textsf{Ker}}(l)$ is a submonoid
of $M$. Moreover, if $M$ is a group, then $\emph{\textsf{Ker}}(l)$ is a subgroup of $M$.
\end{proposition}
\begin{proof}
Since $l(\textsf{1}_M)=\textsf{1}_A$ by (L1), it follows that $\textsf{1}_M \in \textsf{Ker}(l)$. Next, if
$m,m' \in \textsf{Ker}(l)$, then $l(m)=l(m')=\textsf{1}_A$ and we calculate:
$$l(mm')\stackrel{(L3)}=l(m l(m'))=l(m \textsf{1}_M)=l(m)=\textsf{1}_M.$$
Hence $mm' \in \textsf{Ker}(l)$, proving that $\textsf{Ker}(l)$ is a submonoid of $M$.

Assuming additionally that $M$ is a group, we have for any $m \in \textsf{Ker}(l)$:
\[
\begin{split}
\textsf{1}_A \stackrel{(L1)}=l(\textsf{1}_M)&= l(m^{-1} m) \stackrel{(L3)}= l(m^{-1}l(m))\\
&=l(m^{-1} \textsf{1}_M)=l(m^{-1}),
\end{split}
\]
whence $m^{-1} \in \textsf{Ker}(l)$.
Thus, $\textsf{Ker}(l)$ is a subgroup of $M$.
\end{proof}

Recall that a \emph{subgroup} of a monoid $M$ is a subgroup of the group $U(M)$.
The next result determines exactly what is needed to guarantee that a subgroup of a monoid to
be the first factor of a factorization of the monoid.

\begin{theorem}\label{subgroup} Let $M$ be a monoid. A subgroup $L$ of $M$ is the first factor of
a monoid factorization of $M$ if and only if $\mathcal{D}_l(M,L)\neq\emptyset.$ If this is the case
and $q \in \mathcal{D}_l(M,L)$, then the pair $(L, \emph{\textsf{Ker}}(q))$ is a factorization of the monoid $M$.
\end{theorem}
\begin{proof}The condition is necessary by Proposition \ref{l}. If it is satisfied, then there exists
$q \in \mathcal{D}_l(M,L)$. We claim that the pair $(L, \textsf{Ker}(q))$ is a
factorization of the monoid $M$. Indeed, note first that $\textsf{Ker}(q)$ is a submonoid of $M$ by Proposition
\ref{ker.mon}. Next note that for any $m \in M$, $q(m) \in L\subseteq U(M)$ and
\[q(q(m)^{-1}m)\stackrel{(L2)}=q(m)^{-1}q(m)=\textsf{1}_M.\] It follows that
\[q(m)^{-1}m \in \textsf{Ker}(q)\,\,\text{ for all} \,\,m \in M.\]
Now, for every $m \in M$, $m=q(m)(q(m)^{-1}m)$ with $q(m) \in L$ and $q(m)^{-1}m \in \textsf{Ker}(q)$. Thus, $L\textsf{Ker}(q)=M$.
If $l_1k_1=l_2k_2$, where $l_1, l_2 \in L$ and $k_1, k_2 \in \textsf{Ker}(q)$, then $q(k_1)=q(k_2)=\textsf{1}_M$
and we have:
\[l_1=l_1 q(k_1)\stackrel{(L2)}=q(l_1k_1)=q(l_2k_2)\stackrel{(L2)}=l_2q(k_2)=l_2.\]
Since $l_1(=l_2)$ lies in $L$ and it is invertible, the equality $l_1k_1=l_2k_2$ implies
that $k_1=k_2$. Thus, every element of $m \in M$ can be written uniquely as a product $m=lk$ with $l\in L,  k\in \textsf{Ker}(q)$,
proving that $(L, \textsf{Ker}(q))$ is a monoid factorization of $M$.
\end{proof}

\begin{remark}\label{subgroup.r.} It follows from the proof of Theorem \ref{subgroup} that if $L$ is a subgroup of
a monoid $M$ and $q \in \mathcal{D}_l(M,L)$, then the assignment
\[m \longmapsto q(m)^{-1}m\] yields a map $q^\dagger:M \to \emph{\textsf{Ker}}(q)$ which serves as the map $r^{\emph{\textsf{Ker}}(q)}_L$
for the factorization $(L, \emph{\textsf{Ker}}(q))$.
\end{remark}

\begin{theorem}\label{subgroup.t.} Let $M$ be a monoid and $L$ be a subgroup of $M$. The assignment
\[q \mapsto (L,\emph{\textsf{ker}}(q))\]
yields a bijection
\[\mathcal{D}_l(M,L) \simeq \emph{\textsf{FAC}}(L/ M).\] Its inverse takes $(L,B)\in \emph{\textsf{FAC}}(M)$ to $l^B_L$.
\end{theorem}
\begin{proof} According to Theorem \ref{subgroup}, the pair $(L,\textsf{ker}(q))$ is a factorization of $M$.
Moreover, it follows easily from the proof of the theorem that $q^{\textsf{ker}(q)}_L=q$. On the other hand,
for any $(L,B)\in \textsf{FAC}(M)$ we have $\textsf{Ker}(l^B_L)=B$ by Proposition \ref{ker}, which completes the proof.
\end{proof}

Now let the monoid $M$ be factorizable by its submonoids $A$ and $B$
and let $\mathcal{D}_l^{u,B}(M,A)$ be the subset of $\mathcal{D}_l(M,A)$
consisting of all the maps under which the image of $B$ lies in $U(A)$, i.e.,
\[\mathcal{D}_l^{u,B}(M,A)=\{q \in \mathcal{D}_l(M,A):q(B)\subseteq U(A)\}.\]
Since $l^B_A(B)=\{\textsf{1}_A\}\subseteq U(A)$ by Proposition \ref{ker}, $\mathcal{D}_l^{u,B}(M,A)$ is actually a pointed subset of
$\mathcal{D}_l(M,A)$. We write $\D^{u,B}_l(M,A)$ for the pointed set of equivalence classes of such a descent 1-cocycles.
Being the quotient of the set $\mathcal{D}_l(M,A)$ by the restriction of the equivalence relation
on $\mathcal{D}_l(M,A)$ of being equivalent descent 1-cocycles, $\D^{u,B}_l(M,A)$
is a pointed subset of the pointed set $\D_l(M,A)$.

\begin{theorem}\label{com.mon}  In the situation described above, the assignment $q \longmapsto \emph{\textsf{Ker}}(q)$
yields an isomorphism  $\mathcal{D}_l^{u,B}(M,A) \simeq \emph{\textsf{FAC}}(A/M)$ of pointed sets.
\end{theorem}
\begin{proof} Let $q \in \mathcal{D}_l^{u,B}(M,A)$ be an arbitrary element.
Then, by Proposition \ref{ker.mon}, $\textsf{Ker}(q)$ is a submonoid of $M$. Since $M$ is factorized through the
submonoids $A$ and $B$ by the inductive hypothesis, any element $m \in M$ can be written uniquely in the form
$m=l^B_A(m) r^B_A(m)$. In particular, if $k\in \textsf{Ker}(q)$, then
\[\textsf{1}_M=q(k)=q(l^B_A(k) r^B_A(k))\stackrel{(L2)}=l^B_A(k)q(r^B_A(k)).\]
Since $r^B_A(k)\in B$ and hence $q(r^B_A(k)) \in U(A)$ by our assumption on $q$,
it follows that $l^B_A(k)=q(r^B_A(k))^{-1}$.
Thus, any element $k \in \textsf{Ker}(q)$ can be written uniquely in the form
\begin{equation}\label{spl}
  k=q(r^B_A(k))^{-1}  r^B_A(k).
\end{equation} It follows in particular from (\ref{spl}) that if $k,k'\in \textsf{Ker}(q)$ are such that
$r^B_A(k)=r^B_A(k')$, then $k=k'$.

Suppose now $m$ is an arbitrary element of $M$. Since
\[
q[q(r^B_A(m))^{-1}r^B_A(m)]\stackrel{(L2)}=q(r^B_A(m))^{-1} q(r^B_A(m))=\textsf{1}_A,
\]
it follows that $q(r^B_A(m))^{-1} r^B_A(m) \in \textsf{Ker}(q)$.
Evidently,
\[
l^B_A(m)\cdot r^B_A(m)=l^B_A(m) \cdot q(r^B_A(m))\cdot q(r^B_A(m))^{-1}\cdot r^B_A(m)
\]
and $l^B_A(m)\cdot q(r^B_A(m)) \in A$, we can conclude that
\begin{equation}\label{spl.com.mon}
  m=(l^B_A(m) q(r^B_A(m)) (q(r^B_A(m))^{-1} r^B_A(m))
\end{equation}
is an $(A,\textsf{Ker}(q))$-decomposition of $m$.
If $a,a'\in A$ and $k,k'\in \textsf{Ker}(q)$
are such that $ak=a'k'$, then
\[
\begin{split}
a=a\textsf{1}_A&=aq(k)\stackrel{(L2)}=q(ak)\\
&=q(a'k')\stackrel{(L2)}=a'q(k')=a'\textsf{1}_A=a'.
\end{split}
\]
Moreover, we have:
\[
aq(r^B_A(k))^{-1}r^B_A(k)\stackrel{(\ref{spl})}=ak=ak'\stackrel{(\ref{spl})}=aq(r^B_A(k'))^{-1}r^B_A(k').
\]
Since
\begin{itemize}
  \item[-] $aq(r^B_A(k))^{-1}, aq(r^B_A(k'))^{-1}\in A$;
  \item[-] $r^B_A(k), r^B_A(k)\in B$;
  \item[-] $M$ is factorized through $A$ and $B$,
\end{itemize}
it follows that $r^B_A(k)=r^B_A(k')$ and $k=k'$, as we have already remarked.
Thus, each element $m \in M$ is uniquely expressible in the form $m=ak$ with $a \in A$ and $k \in \textsf{Ker}(q)$.
This proves that $M$ is factorized through $A$ and $\textsf{Ker}(q)$, or equivalently,
$\textsf{Ker}(q) \in \textsf{FAC}(A/M)$. Hence the map $q \longmapsto \textsf{Ker}(q)$ is well defined.
Clearly, the map $l^B_A:M \to A$ is the base point of the pointed set
$\mathcal{D}_l^u(M,A)$ and $\textsf{Ker}(l^B_A )= B$ by Proposition \ref{ker}, it follows that
the map $q \longmapsto \textsf{Ker}(q)$ is a morphism of pointed sets. In order to show that this morphism
is an isomorphism, we construct its inverse.

Suppose now that $C \in \textsf{FAC}(A/M)$. Then $M$ is factorized through $A$ and $C$. We claim that the map $q_C=l^C_A:M \to A$
lies in $\mathcal{D}_l^{u,B}(M,A)$. According to Proposition \ref{l}, $q_C$ lies in  $\mathcal{D}_l(M,A)$. Thus we only have to show
that  $q_C(B)\subseteq U(A)$. For this, we consider an arbitrary $b \in B$. Then
\begin{equation}\label{spl.com}
b=l^C_A(b) r^C_A(b).
\end{equation}
Similarly, since $M$ is factorized through $A$ and $B$, and $r^C_A(b)\in C\subseteq M$,
\begin{equation}\label{spl.com.1}
 r^C_A(b)=l^B_A(r^C_A(b)) r^B_A(r^C_A(b)).
\end{equation}
The combination of the last two equalities gives that
$$b=l^C_A(b)l^B_A(r^C_A(b))r^B_A(r^C_A(b)).$$
Quite obviously, $l^C_A(b)l^B_A(r^C_A(b))\in A$ and $r^B_A(r^C_A(b))\in B$. Since  $\textsf{1}_A b=b$ and $(A,B)$ is a factorization of $M$ it follows that
$l^C_A(b)l^B_A(r^C_A(b))=\textsf{1}_A$ and $b=r^B_A(r^C_A(b))$.
Therefore
\[
\begin{split}
r^C_A(b)\stackrel{(\ref{spl.com.1})} =& l^B_A(r^C_A(b)) r^B_A(r^C_A(b))\\
=&\; l^B_A(r^C_A(b))b \stackrel{(\ref{spl.com})}=l^B_A(r^C_A(b)) l^C_A(b) r^C_A(b).
\end{split}
\]
We have $l^B_A(r^C_A(b))l^C_A(b)=\textsf{1}_A$, because $l^B_A(r^C_A(b)) l^C_A(b) \in A, r^C_A(b) \in C$ and
$(A,C)$ is a factorization of $M$.  Thus $q_C(b)=l^C_A(b)$ is invertible with inverse $l^B_A(r^C_A(b))$ for all $b \in B$.
Consequently, $q_C \in\mathcal{D}_l^u(M,A)$.

We claim that the maps $q \mapsto \textsf{Ker}(q)$ and $C \mapsto q_C$ are inverses of each other.
To prove our claim, we have to show that
\begin{itemize}
  \item [(i)]$\textsf{Ker}(q_C)=C$ for all $C \in \textsf{FAC}(A/M)$;
  \item [(ii)] $q\,_{\textsf{Ker}(q)}=q$ for all $q \in \mathcal{D}_l^{u,B}(M,A)$.
\end{itemize}
Since (i) follows at once from Proposition (\ref{lr}), we need only to establish (ii).
Given an arbitrary $q \in \mathcal{D}_l^{u,B}(M,A)$,  we have, for any  $m \in M$,
\[
\begin{split}
m=l^B_A(m) r^B_A(m)&=l^B_A(m) q(r^B_A(m))q(r^B_A(m))^{-1} r^B_A(m)\\
&\stackrel{(L2)}=q(l^B_A(m) r^B_A(m)) q(r^B_A(m))^{-1}  r^B_A(m)\\
&=q(m)q(r^B_A(m))^{-1} r^B_A(m)
\end{split}
\]
with $q(m) \in A$ and $q(r^B_A(m))^{-1} r^B_A(m)\in \textsf{Ker}(q)$. It proves that $l^C_A(m)=q(m)$ for all $m\in M$ and
$q\,_{\textsf{Ker}(q_C)}=q$. Thus, (ii) also holds which completes the proof.
\end{proof}

Let $X$ and $Y$ be subsets of a monoid $M$ and let $A$ be a submonoid of $M$. The subsets $X$ and $Y$ are called $A$-\emph{conjugate} if $X=aYa^{-1}$ for some $a \in U(A)$. When this is the case, one says that $a$ \emph{conjugates} $X$ to $Y$.

\begin{lemma}\label{mcj} Let a monoid $M$ be factorizable by its submonoids $A$ and $B$.
Then for any $A$-conjugate $B'$ of $B$, $M$ is factorized by the submonoids $A$ and $B'$.
Said otherwise, the group $U(A)$ acts by conjugation on the set $\,\,\emph{\textsf{FAC}}(A/M)$.
\end{lemma}
\begin{proof}Suppose that $a_0 \in U(A)$ conjugates $B$ to $B'$, that is, $B'=a_0Ba_0^{-1}$.
Since $M$ is factorized through  its  submonoids $A$ and $B$,
there exist elements $b \in B$ and $a_1\in A$ such that $ma_0=a_1b$ for any $m \in M$ and we have:
$$m=ma_0a_0^{-1}=a_1ba_0^{-1}=(a_1a_0^{-1}) (a_0ba_0^{-1}).$$ Quite obviously, $a_1a_0^{-1} \in A$
and $a_0ax_0^{-1}\in B'$, it follows that any element of $M$ can be  written in the form
$a a_0ba_0^{-1}$ with $a \in A$ and $b\in B$. If
\[
a a_0ba_0^{-1}=a' x_0b'a_0^{-1}, \qquad (a,a'\in A,\quad  b,b'\in B)
\]
then $a a_0b=a' a_0b'$, implying -- again since  $M$ is factorized through $A$ and $B$ -- that $b=b'$ and $a a_0=a' a_0$.
Since $a_0$ is invertible, the last equality implies that $a=a'$. Hence each element $m \in M$ is uniquely expressible in the
form $m=ab'$ with $a \in A$ and $b' \in B'$. Consequently, $M$ is factorized through  its  submonoids $A$ and $B'$.
\end{proof}

By essentially the same proof as for  \cite[Proposition 4.6]{BM} we obtain the following.

\begin{proposition}\label{conj.1}
Two elements of $\mathcal{D}_l^{u,B}(M,A)$ are equivalent if and only if the corresponding elements
of $\emph{\textsf{FAC}}(A/M)$ are conjugate in $M$.
\end{proposition}

\begin{proposition}\label{left.1} For any $a\in U(A)$, the (left) action
\[U(A)\times \mathcal{D}_l(M,A) \to \mathcal{D}_l(M,A),\quad q \mapsto a \star q\]
restricts to a (left) action
\[U(A)\times \mathcal{D}^{u,B}_l(M,A) \to \mathcal{D}^{u,B}_l(M,A).\]
\end{proposition}
\begin{proof}We need only establish that for any $a\in U(A)$
and $q\in \mathcal{D}^{u,B}_l(M,A)$, $a\star q $ also lies in $\mathcal{D}^{u,B}_l(M,A)$, or equivalently,
$(a \star q)(B)\subseteq U(A)$. In order to establish the subset inclusion, we consider an arbitrary element
$b \in B$. Since
\[(a \star q)(b)=q(ba)a^{-1},\]
the element $(a \star q)(b)$ lies in $U(A)$ if and only if $q(ba)\in U(A)$.
Since $q\in \mathcal{D}^{u,B}_l(M,A)$, it follows from Theorem \ref{com.mon} that $(A, \textsf{Ker}(q))$
is a monoid factorization of $M$ and there exist elements $a'\in A$ and $k \in \textsf{Ker}(q)$
such that $b=a'k$.
Since $q$ satisfies (L2) and $k \in \textsf{Ker}(q)$ we have
\[q(b)=q(a'k)=a'q(k)=a'.\]
As $q \in \mathcal{D}^{u,B}_l(M,A)$, it follows that $a' \in U(A)$. Again, since $(A, \textsf{Ker}(q))
\in \textsf{FAC}(M)$, there exist elements $a_1,a_2 \in A$ and $k_1,k_2 \in \textsf{Ker}(q)$ such that
\begin{equation}\label{left.a}
  ka=a_1k_1
\end{equation}
and
\begin{equation}\label{left.b}
  k_1a^{-1}=a_2k_2.
\end{equation}
Then \quad $k\stackrel{(\ref{left.a})}=a_1k_1a^{-1}\stackrel{(\ref{left.b})}=a_1a_2k_2$\quad and
 \begin{equation}\label{left.c}
  k=k_2 \text{\quad and \quad} a_1a_2=\textsf{1}_M
\end{equation} by the uniqueness condition in the definition of monoid factorization. Next, since
\[k_1=k_1a^{-1}a\stackrel{(\ref{left.b})}=a_2k_2a \stackrel{(\ref{left.c})} =a_2ka\stackrel{(\ref{left.a})}=a_2a_1k_1,\]
$a_2a_1=\textsf{1}_M$ again  by the uniqueness property of monoid factorizations. Thus, $a_1,a_2 \in U(A)$.
Taking into account the equalities
\[q(ba)=q(a'ka)\stackrel{(\ref{left.a})}=q(a'a_1k_1)\stackrel{(L2)}=a'a_1q(k)=a'a_1\]
and that $a' \in U(A)$ we have $q(ba)\in U(A)$, as desired.
\end{proof}

Recall that to any left action of a group $G$ on a set $X$ there is naturally associated an \emph{action groupoid}
$X /\!/  G$ with $X$ as the set of objects. A morphism from $x\in X$ to $x'\in X$ is an elements $g \in G$ with
$gx = x'$.

According to Proposition \ref{left.1} and Lemma \ref{mcj} the group $U(A)$ acts from the left on both $\mathcal{D}^{u,B}_l(M,A)$
and  $\textsf{FAC}(A/M)$. Hence we have two groupoids  $\mathcal{D}^{u,B}_l(M,A)/\!/ U(A)$ and $\textsf{FAC}(A/M) /\!/ U(A)$.
The combination of Theorem \ref{com.mon} and Proposition \ref{conj.1} gives:

\begin{theorem}\label{conj.2} Let a monoid $M$ is factorized by its submonoids $A$ and $B$. The assignment
\[
q \longmapsto \emph{\textsf{Ker}}(q)
\]
yields an isomorphism of groupoids
\[\mathcal{D}^{u,B}_l(M,A)/\!/ U(A)\simeq \emph{\textsf{FAC}}(A/M) /\!/ U(A).
\]
Moreover, this isomorphism induces a bijection of sets
\[
\D^{u,B}_l(M, A) \simeq \pi_0(\emph{\textsf{FAC}}(A/M)/\!/U(A)),
\]
where  $\pi_0(\emph{\textsf{FAC}}(A/M)/\!/ U(A))$ is the set of
connected components of the groupoid $\emph{\textsf{FAC}}(A/M) /\!/ U(A)$.
\end{theorem}

\section{SEMI-DIRECT PRODUCT AND NON-ABELIAN COHOMOLOGY OF MONOIDS}\label{comp.}

Let $A$ and $B$ be  monoids.  Recall that a \emph{(left) action of $B$ on $A$} is a left action
$$B \times A\to A, \quad (b,a)\longmapsto {b \star a}$$ of the monoid $B$ on the set $A$
such that:
\begin{itemize}
  \item [(A3)] $ b \star \textsf{1}\!_{A}=\textsf{1}\!_{A}$;
  \item [(A4)] $b \star (a_1  a_2)  =(b \star a_1) (b \star a_2)$
\end{itemize} for all $b \in B$ and all $a,a_1,a_2 \in A.$ Thus, a left action of $B$ on $A$ is a
map \[\star :B \times A\to A \] satisfying Conditions (A1)--(A4).

Note that  Conditions (A3) and (A4) express the fact that for every $b\in B$, the map
$$b \star -: A \to A, \,a \mapsto b \star a $$ is a monoid morphism.
Hence to give a left action of a monoid $B$ on a monoid $A$ is to give a homomorphism
$B \to \textsf{End}(A)$ of monoids, where $\textsf{End}(A)$ is the set of all endomorphisms of the
monoid $A$ which is again a monoid under the usual composition of endomorphisms
as the monoid operation. Given a homomorphism $\phi: B \to \textsf{End}(A)$, one obtains
a (left) action of $B$ on $A$ by setting $b\star a:=\phi(b)(a)$.

When a monoid $B$ acts on a monoid $A$ from the
left, one sometimes says that $A$ is a \emph{(left) $B$-monoid}.

\begin{remark}\label{opp.opp}\em It can be seen immediately that if $ \star :B \times A\to A $
is a monoid action, then $A^{\text{op}}$ becomes a left $B$-monoid via $b \star a^{\text{op}}=
(b \star a)^{\text{op}}$.  This $B$-monoid is called the \emph{opposite} to the $B$-monoid $A$.
\end{remark}

For an arbitrary monoid $B$ acting on a monoid $A$
via a homomorphism $\phi: B \to \textsf{End}(A)$,
a $0$-\emph{cohomology monoid} $\textbf{H}_\phi^0(B,A)$ and a $1$-\emph{cohomology pointed set} $\textbf{H}_\phi^1(B,A)$ were constructed in
\cite{BM} as follows. The \emph{zeroth cohomology of $B$ with coefficients in $A$} is the set
$$\{a \in A:b\star a=a \,\,\text{for all}\,\,b \in B\}.$$
Conditions (A3) and (A4) guarantee that this set is in fact a submonoid of $A$.

Next, the set $\mathcal{Z}_\phi^{1}(B, A)$ of \emph{$1$-cocycles of $B$ with coefficients in $A$} is the set of those maps
$\chi: B \to A$ for which
$$\chi(\textsf{1}_B)=\textsf{1}_A $$ and $$\chi(b_1  b_2)=\chi(b_1) (b_1 \star \chi(b_2))$$
for all $b_1, b_2 \in B$. Clearly $\mathcal{Z}_\phi^{1}(B, A)$ includes at least the \emph{unit $1$-cocycle} which is the map
$$0_{B,A}:B \to A, \,\, b \longmapsto \textsf{1}_A.$$ This map turns
$\mathcal{Z}_\phi^{1}(B, A)$ into a pointed set.

One defines a relation on $\mathcal{Z}_\phi^{1}(B, A)$ by calling two $1$-cocycles $\chi$ and $\chi'$
\emph{cohomologous}, written $\chi \sim\chi'$, if there exists
an invertible element $a_0\in U(A)$ such that $\chi(b) (b \star a_0)=a_0  \chi'(b)$ for all $b \in B$.
It is shown in \cite{BM} that $\sim$ is an equivalence relation on $\mathcal{Z}^{1}(B, A)$.
The resulting set of equivalence classes of 1-cocycles is called the \emph{first cohomology pointed set of $B$
with coefficients in $A$} and it is denoted by $\textbf{H}_\phi^1(B,A)$. Note that $\textbf{H}_\phi^1(B,A)$ is not
in general a group, but is a pointed set with the distinguished element being the class of the map $0_{B,A}$.
We normally omit the subscript $\phi$, when there is no  danger of confusion or when the action is clear from the context.

An important class of examples of monoid factorizations are the ones associated to monoid actions.
For any monoids $A$ and $B$, and a monoid homomorphism \[\phi: B \to \textsf{End}(A),\]
the set of formal products
\[A\phi B=\{ab \,|\, a \in A, b\in B\}\]
carries a monoid structure given by the following data:
\begin{itemize}
  \item[-] the identity is $\textsf{1}_A\textsf{1}_B$;
  \item[-] the multiplication $\bullet$ is defined, for any $a_1,a_2 \in A$ and $b_1,b_2\in B$, by
  \[(a_1 b_1) \bullet (a_2b_2) =a_1  (b_1 \star a_2) b_1 b_2,\] where $\star$ is the (left) action of $B$ on $A$ determined by $\phi$.
\end{itemize}
Then $A\phi B$ is called the \emph{semi-direct product of $A$ by $B$ with action} $\phi$.
It is easy to see that the maps
\[\jmath_A:A \to A\phi B, \,\,a \mapsto a\textsf{1}_B\]
and
\[\jmath_B:B \to A \phi B, \,\, b \mapsto \textsf{1}_Ab\] are  (injective)
homomorphisms of monoids and so $A$ can be identified with the subgroup $A'=\{a\textsf{1}_B: a \in A\}$
of $A\phi B$ by identification of the elements $a$ and $a\textsf{1}_B$.
Similarly, $B$ is identified with the subgroup $B'=\{\textsf{1}_Ab: b \in B\}$ of $A\phi B$ by
identification of $b$ and $\textsf{1}_Ab$. Since for any pair $(a,b)$ in $A \times B$,
\[(a\textsf{1}_B)\bullet (\textsf{1}_Ab)=(a(\textsf{1}_B \star \textsf{1}_A))(\textsf{1}_B b)=(a\textsf{1}_A)(\textsf{1}_B b)=ab,\]
it follows that the multiplication map
\[[\imath_{A'},\imath_{B'}]: A' \times B' \to  A\phi B,\,\, (a\textsf{1}_B,\textsf{1}_Ab)
\longmapsto (a\textsf{1}_B)\bullet (\textsf{1}_Ab)=ab\]
is bijective, or equivalently, each element of $A\phi B$ can be written uniquely as a product $a'\bullet b'$
with $a' \in A',  b'\in B'$. Consequently, $(A', B')$ is a monoid factorization of
$A \phi B$. Then modulo the identifications of $A$, $A'$, $B$ and $B'$, the pair $(A, B)$
becomes a monoid factorization of $A \phi B$. We denote the corresponding maps
\[l^B_A:A\phi B \to A, \,\, ab \mapsto a \]
and
\[r^B_A:A \phi B \to B, \,\, ab \mapsto b,\] respectively by $p_A$ and $p_B$
and call them \emph{projections}. Quite obviously, $p_B$ is a monoid homomorphism.

\bigskip

To simplify calculations, from now, we will identify the monoids $A$ and $B$ with their images
$\jmath_A(A)=A'$ and $\jmath_B(B)=B'$ in $A\phi B$.

Let $M$ be a monoid and let $X$ be a subset of $M$. One says that a submonoid $N$ of $M$ is \emph{left $X$-normal}
(resp. \emph{right $X$-normal}) in $M$ if $xN \subseteq Nx$ (resp. $Nx \subseteq xN$) in $M$ for all $x \in X$.
A submonoid of a monoid is called $X$-\emph{normal}, if it is both right and left $X$-normal.
A factorization $(A,B)$ of $M$ is said to be \emph{left} (resp. \emph{right}) \emph{normal} if $A$ (resp. $B$) is left
(resp. right) $B$-normal (resp. $A$-normal) in $M$.

\begin{lemma}\label{normal} Let $(A,B)$ be a factorization of a monoid $M$ such that $A$ is a group. Then
$A$ is left $B$-normal in $M$ if and only if it is left $M$-normal in $M$.
\end{lemma}
\begin{proof} One direction is immediate from the definition. For the converse, we note first that
since $A$ is group, we have $aA=A=Aa$ for all $a \in A$.
Now let $m$ be an arbitrary element of $M$. Then $m=l(m)r(m)$ with $l(m)\in A$ and $r(m) \in B$ and we have:
\[
\begin{split}
mA=l(m)r(m)A&\subseteq l(m)Ar(m)\\
&=A r(m)=(Al(m)^{-1})l(m)r(m)=Am.
\end{split}
\]
Here the inclusion follows from the fact that $A$ is left $B$-normal in $M$. Thus, $mA\subseteq Am$ for all
$m \in M$, proving $M$-normality of $A$.

\end{proof}

\begin{theorem}\label{semi-d.}  For a factorization $(A,B)$ of a monoid $M$ the following sentences are equivalent:
\begin{enumerate}
  \item [(i)] $r=r^B_A:M \to B$ is a homomorphism of monoids.
  \item [(ii)] $(A,B)$ is left normal.
  \item [(iii)] The map $\phi=\phi^B_A:B \to \emph{\textsf{End}}(A), \,\, \phi(b)=(a \to l(ba))$ is an action
  of the monoid $B$ on the monoid $A$ for which $M$ is isomorphic as a monoid to $A {\phi} B$ via
  the map $A \phi B \xr{(ab \mapsto \imath_A(a)\imath_B(b)} M$.
\end{enumerate}
\end{theorem}

\begin{proof} First we prove that (i) and (ii) are equivalent. If $r:M \to B$ is a homomorphism of monoids, then for any $a \in A$ and $b \in B$,  we have:
$$r(ba)=r(b)r(a)\stackrel{(\ref{lr})}=r(b)\textsf{1}_M=r(b)\stackrel{(R1)}=b.$$ It then  follows
that $ba=l(ba)r(ba)=l(ba)b$, implying -- since $a\in A$ and $b \in B$ were arbitrary and since $l(ba)\in A$ --
that $bA\subseteq Ab$ for all $b \in B$. Thus, $A$ is left $B$-normal in $M$, or equivalently, $(A,B)$ is left normal.

Conversely, suppose that $A$ is left $B$-normal in $M$. Since $r(\textsf{1}_M)=\textsf{1}_M$ by (R1), $r$ preserves units.
Next, consider arbitrary two elements $m,m'\in M$.  Then $m=l(m)r(m)$ and $m'=l(m')r(m')$ and hence
$mm'=l(m)r(m)l(m')r(m')$. Since $r(m)l(m')\in r(m)A$ and $r(m)A \subseteq Ar(m)$ by $B$-normality of $A$,
it follows that $r(m)l(m')=a\cdot r(m)$ for some $a \in A$. Then
\[mm'=l(m)\cdot a \cdot r(m)\cdot r(m')\] implying --
since $l(m)\cdot a \in A$ (and hence $r(l(m)\cdot a)=\textsf{1}_M$ by (\ref{lr})) and
$r(m)\cdot r(m')\in B$ -- that
\[
\begin{split}
r(mm')=&r(l(m)\cdot a \cdot r(m)\cdot r(m'))\\
 \stackrel{(R2)}=&r(l(m)a)\cdot r(m)\cdot r(m')\\
 =&\textsf{1}_B \cdot r(m)\cdot r(m')=r(m)\cdot r(m').
\end{split}
\]
Consequently, $r$ is a monoid homomorphism.

(ii)$\Longrightarrow$(iii). We have to show that the monoid action corresponding to the map $\phi^B_A$ satisfies Conditions from (A1) to (A4).

Since for any $a \in A$, $\textsf{1}\!_{B}\star a=l(\textsf{1}\!_{M}a)=l(a)\stackrel{(L1)}=a$, (A1) holds. Similarly, since
for any $b \in B$, $b\star \textsf{1}\!_{A}=l(b\textsf{1}\!_{M})=l(b)\stackrel{(\ref{lr})}=\textsf{1}\!_{A}$, (A3) also holds.

Next, for any $b_1,b_2 \in B$ and $a\in A$, we have:
\[
\begin{split}
(b_1 b_2)\star a=l((b_1 b_2)a)=&l (b_1 (b_2a))\stackrel{(L3)}=l(b_1 l(b_2a))\\
=&l(b_1 (b_2 \star a))=b_1 \star (b_2 \star a),
\end{split}
\]
where the first, fourth and fifth equalities follow from the definition of the action, while the second
one holds by associativity of multiplication in $M$. Thus, (A2) holds.

Finally, to prove that (A4) also holds, we proceed as follows. For any $b \in B$ and $a_1,a_2 \in A,$
consider $b \star (a_1  a_2)=l(ba_1  a_2)$. Since $A$ is left $B$-normal in $M$, $ba_1=a'b$ for some
$a' \in A$. Then clearly $l(ba_1)=a'$ and we have:
\[
\begin{split}
b \star (a_1  a_2)=l(ba_1  a_2)=&l(a'ba_2)\stackrel{(L2)}=a'\cdot l(ba_2)\\
=&l(ba_1)\cdot l(ba_2)=(b \star a_1) (b \star a_2).
\end{split}
\]
This proves that (A4) also holds.

(iii)$\Longrightarrow$(i). We have already pointed out that for the semi-direct product $A {\phi} B$, the 
map $p_B=r^B_A: A {\phi} B \to B$ is a homomorphism of monoids.
\end{proof}

As an immediate consequence of the above theorem we observe that

\begin{corollary}\label{semi-d.1} A monoid $M$ is (isomorphic as a monoid to) a semi-direct product of
monoids if and only if there exists a left normal factorization of $M$.
\end{corollary}

One calls a homomorphism $p:M \to B$ of monoids \emph{split} if there is some monoid homomorphism $s:B \to M$ with
$ps=\textsf{Id}_B$.

\begin{lemma}\label{sch.} For a split homomorphism of monoids $\xymatrix{M \ar@<0.5ex>[r]^-{p} & B \ar@<0.5ex>[l]^-{s},\, ps=\emph{\textsf{Id}}_B,}$
the following sentences are equivalent:
\begin{itemize}
  \item [(i)] The set $\emph{\textsf{Ker}}(p)$ is a group and $(\emph{\textsf{Ker}}(p),s(B))$ is a monoid factorization of $M$.
  \item [(ii)] Let $m_1,m_2 \in M$. If $p(m_1)=p(m_2)$, then there exists a unique $k\in \emph{\textsf{Ker}}(p)$
  with $m_2=km_1$.
\end{itemize} Under either Conditions \emph{(i)} or\emph{ (ii)}, $\emph{\textsf{Ker}}(p)$ is left $s(B)$-normal in $M$.
\end{lemma}
\begin{proof} (i)$\Longrightarrow $(ii). Suppose that $m_1,m_2 \in M$ are such that $p(m_1)=p(m_2)$. Since
the pair $(\textsf{Ker}(p),s(B))$ is a monoid factorization of $M$ by the inductive hypothesis, there are some $k_1,k_2 \in \textsf{Ker}(p)$
such that $m_1=k_1s(b),\, m_2=k_2s(b)$, where $b$ is the common value of $p(m_1)$ and $p(m_2)$.
Then $m_2=km_1$, where $k=k_2k_1^{-1} \in \textsf{Ker}(p)$.
If $k'\in \textsf{Ker}(p)$ is such that $m_2=k'm_1$, then $k_2 s(b)=k'k_1s(b)$ and $k_2=k'k_1$ by the uniqueness
property of monoid factorizations. Therefore, $k'=k_2k_1^{-1}=k$.

(ii)$\Longrightarrow $(i). It is easy to show that under Condition (ii), $\textsf{Ker}(p)$ is a group.
So it suffices to show that the pair $(\textsf{Ker}(p),s(B))$ is a monoid factorization of $M$.
Since for any $m\in M$, $p(m)=p(sp(m))$, there exist a unique element $k \in \textsf{Ker}(p)$ with $m=k \cdot sp(m).$
Next, if $k_1 \cdot s(b_1)=k_2 \cdot s(b_2)$ with $k_1,k_2 \in \textsf{Ker}(p)$ and $b_1,b_2 \in B$, then
\[b_1=p(k_1 \cdot s(b_1))=p(k_2 \cdot s(b_2))=b_2.\]
Therefore $s(b)=k^{-1}_1k_2s(b)$
and $k^{-1}_1k_2=\textsf{1}_M$ by the uniqueness condition in (ii).
Hence $k_1=k_2$. Consequently, any element of $M$ can be written uniquely as a product of an element
of $\textsf{Ker}(p)$ and an element of $s(B)$, proving that $(\textsf{Ker}(p),s(B))$ is a monoid factorization
of $M$.

Finally, assuming Condition (ii), if $k \in \textsf{Ker}(p)$ and $b \in B$ are arbitrary elements, then since
\[p(s(b)\cdot k)=p(s(b))\cdot p(k)=p(s(b)),\]
there exists a unique element $k'\in \textsf{Ker}(p)$ with $s(b)\cdot k=k'\cdot s(b)$. It follows that
\[s(b)\cdot\textsf{Ker}(p) \subseteq \textsf{Ker}(p)\cdot s(b) \,\,\text{for all}\,\, b \in B, \]
proving that $\textsf{Ker}(p)$ is left  $s(B)$-normal.
\end{proof}

Given a submonoid $A$ of a monoid $M$, we call a descent 1-cocycle $q\in \mathcal{D}_l(M,A)$ \emph{left normal}
if $A$ is left $\textsf{Ker}(q)$-normal.

\begin{theorem}\label{bijection} For a monoid $M$, the following  sets are in bijective correspondence:
\begin{itemize}
  \item [(i)] the set of left normal factorizations of $M$ whose first factor is a group;
  \item [(ii)] the set of left normal descent 1-cocycles with domain $M$ and codomain a group, and
  \item [(ii)] the set of split epimorphisms of monoids with domain $M$ satisfying Condition (ii) of Lemma \ref{sch.}.
\end{itemize}These bijective correspondences are explicitly stated in the following table.
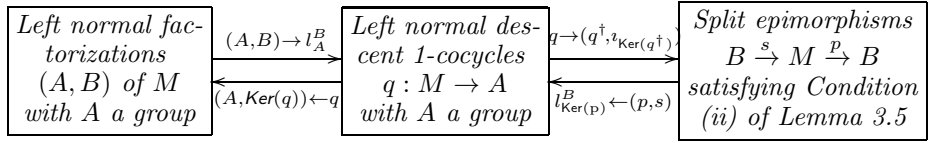
\[\xymatrix @R=.4in @C=.3in{
 *+[F]\txt{\text{Left normal fac-}\\\text{torizations } \\ $(A,B)$ of $M$  \\ with  $A$ a group} \ar@<1ex>[rr]^-{(A,B) \to \, l^B_A}  &&
*+[F]\txt{\text{Left normal des-}\\\text{cent 1-cocycles }\\ $q:M \to A$ \\with $A$ a group}
 \ar@<1ex>[ll]^-{\, (A, \textsf{Ker}(q))\leftarrow q\,} \ar@<1ex>[rr]^-{\,\,\,q \to (q^\dag \!, \imath_{\emph{\textsf{Ker}}(q^\dag)})\,\,}
 &&*+[F]\txt{\text{Split epimorphisms}\\ $B \xr{s}M \xr{p} B $\\  satisfying Condition  \\(ii) of Lemma \ref{sch.}}
 \ar@<1ex>[ll]^-{\,l_{\emph{\textsf{Ker}(p)}}^B \leftarrow (p,s)\,}}\]
\end{theorem}
\begin{proof} The bijection between (i) and (ii) follows from Proposition \ref{l}, Theorems \ref{subgroup} and \ref{semi-d.}.
The correspondences between (ii) and (iii) are given by Theorem \ref{semi-d.} and  Lemma \ref{sch.}.
\end{proof}

Let $B$ be a monoid acting on another monoid $A$ via a homomorphism $\phi: B \to \textsf{End}(A)$
and $ A\phi B$ the corresponding semidirect product. It is easy to see that any section
$f:B \to  A\phi B$ of the projection $p_B=l^B_A:A\phi B \to B$ (i.e.  a homomorphism $f$ of monoids  with
$p_B f=\textsf{Id}_B$)  has the form $f(b)=\ol{f} (b)b$, where $\ol{f}: B \to A$ is
the composite $p_Af$. Two sections $f$ and $g$ are called \emph{equivalent} if
there is an invertible element $a_0 \in A$ such that $a_0f(b)a_0^{-1}=g(b)$
for all $b \in B$. In other words, two sections are equivalent if they differ by conjugation with an invertible element of $A$.
(Recall that we have identified $A$ with its image $A\textsf{1}_B$ in $A\phi B $.)

The following proposition can be proved in a similar was as \cite[Proposition 2.3]{Brw}.

\begin{proposition}\label{spl.}  In the situation described above,  the map that sends
a $1$-cocycle $\chi:B \to A$ to the map \[\ol{\chi}:B \to A\phi B,\,\, b \mapsto \chi(b)b\] establishes a
bijection of $\mathcal{Z}^1(B,A)$ (resp. $\emph{\textbf{H}}^1(B,A)$)  with the set (of equivalence classes) of sections
of the projection $p_B:A\phi B \to B$. Its inverse takes (the class of) a section $\omega$ of $p_B$ to (the
class of) the composite $\ol{\omega}=p_A\omega$.
\end{proposition}

It was shown in \cite{BM} that a (left) action of a monoid $B$ on a monoid $A$ restricts to a (left) action of
$B$ on the group $U(A)$ and thus we can consider the pointed sets $\mathcal{Z}^1(B, U(A))$ and $\textbf{H}^1(B, U(A))$.

\begin{theorem}\label{spl.com.}  Let $B$ be a monoid, let $A$ be a left $B$-monoid via a homomorphism
$\phi: B \to \emph{\textsf{End}}(A)$ of monoids and let $A \phi B$ be the corresponding semi-direct product. The assignment
\[\chi \longmapsto \{\,\chi(b)b\,:\, b \in B\,\}\subseteq A \phi B\]
establishes an isomorphism $\mathcal{Z}^1(B, U(A)) \simeq \emph{\textsf{FAC}}(A/A\phi B)$ of pointed sets.
\end{theorem}
\begin{proof}First let us observe that the pointed set  $\mathcal{Z}^1(B, U(A))$
can be identified with the pointed subset of $\mathcal{Z}^1(B, A)$ containing those 1-cocycles
$\chi:B \to A$ that factor through $U(A)$. Write $\mathcal{Z}_u^1(B, A)$ for this pointed subset.

Next, the assignment that takes
$\chi\in\mathcal{Z}^1(B, A^{op})$ to the map
\[\hat{\chi}: A\phi B \to A, \,\, \hat{\chi}(ab)=a\chi(b)\]
induces an isomorphism
\begin{equation}\label{spl.eq.}
\mathcal{Z}^1(B, A^{op}) \simeq \mathcal{D}_l(A\phi B,A)
\end{equation}
of pointed sets  by \cite[Proposition 5.5]{BM}. If $\chi\in\mathcal{Z}^1(B, A^{op})$ is such that $\chi(B)\subseteq U(A^{op})$,
then
\[\hat{\chi}(\textsf{1}_AB)=\chi(B)\subseteq U(A^{op})\]
and $\hat{\chi} \in \mathcal{D}^{u,B}_l(A\phi B,A)$. It follows that the isomorphism
of pointed sets (\ref{spl.eq.}) restricts to an isomorphism $\mathcal{Z}_u^1(B, A^{op}) \simeq \mathcal{D}^{u,B}_l
(A\phi B,A)$ of pointed sets and thus one has commutativity in
the following diagram
\[\xymatrix @R=.4in @C=.5in { \mathcal{Z}^1(B, U(A)) \ar[r]^{\simeq}_{\chi \to \chi^{-1}}& \mathcal{Z}^1(B, U(A^{op}))
\ar@{^{(}->}[d]\ar[r]^{\simeq}_{\chi \to \hat{\chi}} & \mathcal{D}^{u,B}_l(A\phi B,A) \ar@{^{(}->}[d]\\
&\mathcal{Z}^1(B, A^{op}) \ar[r]^{\simeq}_{\chi \to \hat{\chi}} & \mathcal{D}_l(A\phi B,A)}\]
in which the vertical maps are subset inclusions, while the map $\chi \mapsto \chi^{-1}$ takes $\chi:B \to U(A)$ to the map
$\chi^{-1}:B \to U(A^{op})$ defined by $\chi^{-1}(b)=(\chi(b))^{-1}$.
Since, $\mathcal{D}^{u,B}_l(A\phi B,A)\simeq \textsf{FAC}(A/A\phi B)$ as
pointed sets by Theorem \ref{com.mon}, we have the following string of isomorphisms of pointed sets:
\[\mathcal{Z}^1(B, U(A)) \simeq  \mathcal{D}^{u,B}_l(A\phi B,A)\simeq \textsf{FAC}(A/A\phi B),\]
which takes ${B \xr{\chi} U(A)}$ to the kernel of the map
\[A \phi B \xr{ab\mapsto a\chi(b)^{-1}} A,\] which is just $
\{\,\chi(b)b)\,:\, b \in B\,\}$, as it can be easily verified. This completes the proof.
\end{proof}

As an immediate consequence we obtain the following theorem.

\begin{theorem}\label{T-compl.mon.} Let $M=A\phi B$, where $\phi:B \to \emph{\textsf{End}}(A)$ is a fixed homomorphism of monoids.
The assignment $\chi \longmapsto \emph{\textsf{Im}}(\ol{\chi})$,
where the map $\ol{\chi}:A \phi B  \to A$ is defined by $\ol{\chi}(ab)=a\chi(b)^{-1}$, yields an isomorphism of groupoids
$$ \mathcal{Z}^{1}(B, U(A))/\!/U(A)\simeq \emph{\textsf{FAC}}(A/A\phi B) /\!/ U(A),$$ which in turn induces an isomorphism of pointed sets
$$\emph{\textbf{H}}^1(B,U(A)) \simeq \pi_0(\emph{\textsf{FAC}}(A/A\phi B)\!/U(A)),$$ were $\pi_0(\emph{\textsf{FAC}}(A/A\phi B)\!/U(A))$
is the set of the isomorphism classes of objects of the groupoid $\emph{\textsf{FAC}}(A/A\phi B)\!/U(A)$.
\end{theorem}

\section{APPLICATIONS}
In this section we present certain applications for the results obtained in the previous sections.

Recall that a monoid $A$ is \emph{conical} if $\textsf{1}_A$ is the only invertible element in $A$, i.e.,
$U(A)=\{\textsf{1}_A\}$. The basic example of a conical monoid is the set of non-negative integers $\N_0$,
under the addition operation. Other interesting examples can be obtained from unital rings, as in the
following example:

\begin{example} For a ring $R$, the set $K(R)$ of isomorphism types of finitely generated projective left $R$-modules form a
(commutative) monoid with unit $[0]$ under the operation $[M] + [M'] := [M\oplus M']$.
It is clear from the definition that $K(R)$ is conical.
\end{example}

\begin{theorem}\label{conical} Let $A$ be a submonoid of a monoid $M$. If $A$ is conical, then $|\emph{\textsf{FAC}}(A/M)|\leq 1$.
\end{theorem}
\begin{proof}If $\textsf{FAC}(A/M)=\varnothing $, then clearly
$|\textsf{FAC}(A/M)|=0$. Otherwise there exists a submonoid $B$ of $M$ such that $M$ is factorized through $A$
and $B$ and it follows from Theorem \ref{com.mon}  that $\mathcal{D}_l(M, U(A)) \simeq \textsf{FAC}(A/M)$.
However, $U(A)=\{\textsf{1}_A\}$, because $A$ is assumed to be conical. Thus, $\mathcal{D}_l(M, U(A))$ (and hence also
$\textsf{FAC}(A/M)$) is one-point set. Therefore, $|\textsf{FAC}(A/M)|= 1$.
\end{proof}

A corollary follows immediately.

\begin{corollary}\label{conical.1} Let $M$ be a monoid. If there exits a monoid factorization $(A,B)$
of $M$ with $A$ conical, then $\emph{\textsf{FAC}}(A/M)=\{B\}$.
\end{corollary}

Suppose that $\kappa: B \to A$ is a homomorphism of monoids such that $\kappa (B)\subseteq U(A)$. The following identity
\[\phi_\kappa(b)(a)=\kappa(b)a \kappa(b)^{-1}, \qquad (a \in A, \,b\in B)\] defines a homomorphism $\phi_\kappa:B
\to \textsf{End}(A)$ of monoids (and hence a left monoid action of $B$ on $A$).
In this case, one says that the \emph{action of $B$ on $A$ is defined by the homomorphism $\kappa$}.

\begin{theorem}\label{action.phi} In the situation described above, the assignment
\[\chi \longmapsto  \chi \ast \kappa,\]  where $\chi \ast \kappa$ is the \emph{convolution product} of
$\chi$ and $\kappa$ given by $(\chi \ast \kappa)(b)=\chi(b)\kappa(b)$, yields
an isomorphism of pointed sets \[\mathcal{Z}^1_{\phi_\kappa} (B,A) \simeq \emph{\textsf{Hom}}(B,A),\]
in which $\emph{\textsf{Hom}}(B,A)$ is a pointed set with distinguished
element given by the homomorphism $\kappa$.
\end{theorem}
\begin{proof} Direct inspection shows that the map
\[T:A \phi_\kappa B \to A \times B, \,\,ab \mapsto (a\kappa(b),b)\]
is an isomorphism of monoids, with inverse $(a,b) \to (a\kappa(b)^{-1})b$. Here $A \times B$ is the direct
product of the monoids $A$ and $B$. Quite obviously, $T$ makes the diagram
\[\xymatrix @R=.4in @C=.5in {
 A \phi_\kappa B \ar[rr]^T \ar[dr]_{P_B} && A \times B\ar[dl]^{P_B}\\
& B&}\]
commutive. According to Proposition \ref{spl.}, the assignment that takes $\chi \in \mathcal{Z}^1_{\phi_\kappa} (B,A)$ to
the map \[\ol{\chi}:B \to A \phi_\kappa  B,\,\, b \mapsto \chi(b)b\] yields a
bijection between $\mathcal{Z}^1(B,A)$ and the set of sections of the projection $p_B:A \phi_\kappa B \to B$.
Moreover, since $T$ is an isomorphism of monoids and since the diagram commutes, the rule $\sigma \to T\sigma$ establishes
a one-to-one correspondence between the sets of sections of the projections $p_B:A\phi_\kappa B \to B$
and $p_B:A\times B \to B$. Since $A \times B$ is the direct product of the monoids $A$ and $B$,
to give a section of the projection $p_B:A\times B \to B$ is to give a homomorphism $B \to A$
of monoids. Consequently, the assignment $\chi \longmapsto  P_A T \ol{\chi}\,\,$ yields
a bijection  \[\mathcal{Z}^1_{\phi_\kappa} (B,A) \simeq \textsf{Hom}(B,A)\] of sets. We write $\Phi$ for this bijection. Since for any
$\chi \in \mathcal{Z}^1_{\phi_\kappa} (B,A)$ and any $b \in B$, one has
\[
\begin{split}
\Phi(\chi)(b)=(P_A T \ol{\chi})(b)=&(P_A T )(\chi(b)b)\\
=&P_A((\chi(b) \kappa(b))b)=\chi(b) \kappa(b)=(\chi \ast \kappa)(b),
\end{split}
\]
we have \[\Phi(\chi)=\chi \ast \kappa \,\,\,\text{for all}\,\,\, \chi \in \mathcal{Z}^1_{\phi_\kappa} (B,A)\]
and the result follows by $\Phi(0_{B,A})=\kappa$.
\end{proof}

Define an equivalence relation on $\textsf{Hom}(B,A)$ by $f \approx g$ if and only if
there exists an element $a \in U(A)$ such that $g(b)=af(b)a^{-1}$ for all $b \in B$.
It is easy to see that $\approx$ is an equivalence relation. Write $\textsf{\textbf{Hom}}(B,A)$
for the quotient pointed set $\textsf{Hom}(B,A)/\approx$.

\begin{theorem}\label{action.phi.1} In the situation of Theorem \ref{action.phi}, the assignment
\[[\chi] \longmapsto  [\chi \ast \kappa]\]
yields
an isomorphism of pointed sets \[\emph{\textbf{H}}^1_{\phi^\kappa}(B,A) \simeq \textsf{\textbf{Hom}}(B,A).\]
\end{theorem}
\begin{proof} According to Theorem \ref{action.phi}, it suffices to show that if $[\chi]=[\chi']$ in $\textbf{H}^1_{\phi^\kappa}(B,A)$,
then $[\chi \ast \kappa]=[\chi' \ast \kappa]$ in $\textsf{\textbf{Hom}}(B,A)$. Suppose that
$[\chi]=[\chi']$ in $\textbf{H}^1_{\phi^\kappa}(B,A)$. Then (see Section \ref{comp.}) there exists an element $a \in U(A)$ such that
$ a\chi'(b)=\chi(b) (b\star a)$ for all $b\in B$. Since $b\star a=\kappa(b) a \kappa(b)^{-1}$, it follows that
\begin{equation}\label{action.eq.}
  \chi'(b)=a^{-1}\chi(b)\kappa(b) a \kappa(b)^{-1} \,\,\text{for all}\,\, b \in B.
\end{equation} We then have for all $b \in B$:
\begin{align*}
(\chi' \ast \kappa)(b) &=&\text{by definition of} \,\,\ast\\
&=\chi'(b) \kappa(b)&\text{by} (\ref{action.eq.})\\
&=a^{-1}\chi(b)\kappa(b) a \kappa(b)^{-1}\kappa(b)\\
&=a^{-1}\chi(b)\kappa(b) a &\text{by definition of} \,\,\ast\\
&= a^{-1}(\chi \ast \kappa)(b) \,a.
\end{align*} This proves that $\chi \ast \kappa$ and $\chi' \ast \kappa$ are equivalent in $\textsf{Hom}(B,A)$
and hence $[\chi \ast \kappa]=[\chi' \ast \kappa]$ in $\textsf{\textbf{Hom}}(B,A)$.

\end{proof}

\bibliographystyle{amsplain}

\begin{thebibliography}{3}
\bibitem{Brw} K.S. Brown, {\em Cohomology of groups}, Corrected reprint of the 1982 original.
Graduate Texts in Mathematics, \textbf{87}. Springer-Verlag, New York, 1994.

\bibitem{BM} V. Bovdi and B. Mesablishvili, {\em Descent cohomology and factorizations of groups}, Alg. Represent. Theory (to appear)

\bibitem{Dummit} D. Dummit, R. Foote, {\em Abstract algebra,} Third edition. John Wiley \& Sons, Inc., Hoboken, NJ, 2004.

\bibitem{Me} B. Mesablishvili, {\em On descent cohomology}, { Transactions of A. Razmadze Mathematical Institute} \textbf{173} (2019), 137--155.

\bibitem{Mes} T. Mesablishvili, {\em On factorization of monoids}, Rep. Enlarged Sess. Semin. I. Vekua Appl. Math. \textbf{35} (2021), 67--70.

\end{thebibliography}

\end{document}